\documentclass[12pt,final]{article}
\usepackage[margin=2.5cm,top=1.5cm]{geometry}
\usepackage[all,arc]{xy}
\usepackage{mathrsfs}
\usepackage{amsmath,amsthm,amssymb,enumerate}
\usepackage{color}

\newtheorem{definition}{Definition}[section]
\newtheorem{lemma}[definition]{Lemma}
\newtheorem{theorem}[definition]{Theorem}

\newtheorem{proposition}[definition]{Proposition}

\newtheorem{remark}[definition]{Remark}

\newcommand{\Romannum}[1]{\uppercase\expandafter{\romannumeral #1}}

\numberwithin{equation}{section}

\newcommand\keywordsname{Key words}
\newcommand\AMSname{AMS subject classifications}

\newenvironment{@abssec}[1]{%
     \if@twocolumn
       \section*{#1}%
     \else
       \vspace{.05in}\footnotesize
       \parindent .2in
         {\upshape\bfseries #1. }\ignorespaces
     \fi}
     {\if@twocolumn\else\par\vspace{.1in}\fi}

\begin{document}

\title{On the exponent set of nonnegative primitive tensors
\footnote{P. Yuan's research is supported by the NSF of China (Grant No. 11271142) and
the Guangdong Provincial Natural Science Foundation(Grant No. S2012010009942),
L. You's research is supported by the Zhujiang Technology New Star Foundation of
Guangzhou (Grant No. 2011J2200090) and Program on International Cooperation and Innovation, Department of Education,
Guangdong Province (Grant No.2012gjhz0007).}}
\author{ Zilong He\footnote{{\it{Email address:\;}}hzldew@qq.com.}
 \qquad Pingzhi Yuan\footnote{{\it{Corresponding author:\;}}yuanpz@scnu.edu.cn.}
 \qquad Lihua You\footnote{{\it{Email address:\;}}ylhua@scnu.edu.cn.}
 }
\vskip.2cm
\date{{\small
School of Mathematical Sciences, South China Normal University,\\
Guangzhou, 510631, P.R. China\\
}}
\maketitle

\begin{abstract} In this paper, we present a necessary and sufficient condition for a nonnegative  tensor to be a primitive one,
 show that the exponent set of nonnegative primitive tensors with order $m(\ge n)$ and dimension $n$ is $\{k| 1\le k\le (n-1)^2+1\}. $
\end{abstract}

\section{Introduction}
\hskip.6cm An $n\times n$ nonnegative square matrix $A= (a_{ij})$ is nonnegative primitive (or simply, primitive) if $A^k>0$ for
some positive integer $k$. The least such $k$ is called the primitive exponent (or simply, exponent) of $A$ and is
denoted by $\gamma(A)$.

Let $a,b,n$ be positive integers with $b>a$, $[a,b]^o=\{k | k \mbox { is an integer and } a\leq k\leq b\}$,
$[n]=[1,n]^o=\{1,2,\ldots, n\}$,
 $E_n=\{k | \mbox { there exists a   primitive matrix }  A \mbox { of order } n \mbox { such that } $ $\gamma(A)=k\}.$

   In 1950, H. Wielandt \cite{Wi59} first stated the sharp upper bound for
$\gamma(A)$, that is, $\gamma(A)\le w_n = (n-1)^2+1$ for all  $n\times n$  primitive matrices and showed that $E_n\subseteq [1,w_n]^o$.
In 1964, A. L. Dulmage and N. S. Mendelsohn \cite{DM64}  revealed the so-called gaps in
the exponent set of  primitive matrices, that is, $E_n \subset [1,w_n]^o$,
where  ``gap" is a set  of consecutive integers $[a,b]^o(\subset [1,w_n]^o)$,  such that no  $n\times n$ matrix $A$ satisfying $\gamma(A)\in [a,b]^o$.
In 1981, M. Lewin and Y. Vitek \cite{LV81}  found all gaps in $[\lfloor\frac{1}{2}w_n\rfloor+1, w_n]^o$,
and  conjectured that $[1, \lfloor\frac{1}{2}w_n\rfloor]^o$ has no gaps,
where $\lfloor x\rfloor$ denotes the greatest integer $\le x$.
 In 1985, Shao  \cite{Sh85} proved that this Lewin-Vitek Conjecture is true for all sufficiently large $n$,
and the conjecture has one counterexample when $n=11$ since $48\not\in E_{11}$.
Finally, in 1987, Zhang \cite{Zh87} continued and completed the work. He  showed that the  Lewin-Vitek Conjecture holds for all $n$
 except $n=11$. Thus the  exponent set $E_n$ for primitive matrices of order $n$ is completely determined.

Since the work of Qi \cite{Qi05} and Lim \cite{Li05}, the study  of tensors and the spectra of tensors (and hypergraphs) and their various applications have attracted much attention and interest.  In \cite{Ch1} and \cite{Ch2}, Chang et al.  defined the irreducibility of tensors, the primitivity of nonnegative tensors, and extended many important properties of (nonnegative) primitive matrices to (nonnegative) primitive tensors. As an application of the general tensor product defined by Shao \cite{Sh12}, Shao presented a simple characterization of the primitive tensors in terms of the zero pattern of the powers of  $\mathbb{A}$. He also proposed a conjecture on the primitive degree $\gamma(\mathbb{A})$. Recently,  the authors \cite{YHY13} confirmed the conjecture of Shao by proving Theorem \ref{thm11}.

\begin{theorem}\label{thm11}{\rm (\cite{YHY13} Theorem 1.2)}
Let $\mathbb{A}$ be a nonnegative primitive tensor with order $m$ and dimension $n$.
Then its primitive degree $\gamma(\mathbb{A})\le (n-1)^2+1$, and the upper bound is sharp. \end{theorem}

Therefore, it is natural to consider the question to completely determine the exponent set of  primitive tensors with order $m$ and dimension $n$.
Let $m(\ge2), n$ be  positive integers and
$$E(m,\,n)=\{k | \mbox { there exists a primitive tensor } \mathbb{A} \mbox { of order } m \mbox{ and  dimension } n \mbox {  such that  }
  k=\gamma(\mathbb{A}) \}.$$

The main result of this paper is as follows.

\begin{theorem}\label{thm12} Let  $m, n$ be positive integers with $m\ge n\ge 3$. Then $E(m,\,n)=[1, (n-1)^2+1]^o$. \end{theorem}
The above theorem shows that there are no gaps in tensor case when $m\ge n\ge3$. It is well-known that there exist gaps when $m=2$,
therefore it is an interesting open problem to determine whether there exist gaps when $3\le m<n$.

The arrangement of the paper is as follows. In Section 2, we will give some definitions and  properties.
In Section 3, we present a necessary and sufficient condition for a nonnegative  tensor to be a primitive one (Theorem \ref{thm34}).
In Section 4,  We prove the main result (Theorem \ref{thm12}).

\section{Preliminaries}
\hskip.6cm An order $m$ dimension $n$ tensor $\mathbb{A}= (a_{i_1i_2\cdots i_m})_{1\le i_j\le n \hskip.2cm (j=1, \ldots, m)}$ over the complex field $\mathbb{C}$ is a multidimensional array with all entries $a_{i_1i_2\cdots i_m}\in\mathbb{C}\, ( i_1, \ldots, i_m\in [n]=\{1, \ldots, n\})$. The majorization matrix  $M(\mathbb{A})$ of the tensor $\mathbb{A}$ is defined as  $(M(\mathbb{A}))_{ij}=
a_{ij\cdots j} (i, j\in[n])$ by Pearson \cite{Pe10}.

 Let $\mathbb{A}$ (and $\mathbb{B}$) be an order $m\ge2$ (and $k\ge 1$), dimension $n$ tensor, respectively. In \cite{Sh12}, Shao defines a general product  $\mathbb{A}\mathbb{B}$ to be the following tensor $\mathbb{D}$ of order $(m-1)(k-1)+1$ and dimension $n$:
$$ d_{i\alpha_1\cdots\alpha_{m-1}}=\sum\limits_{i_2, \ldots, i_m=1}^na_{ii_2\cdots i_m}b_{i_2\alpha_1}\cdots b_{i_m\alpha_{m-1}} \quad (i\in[n], \, \alpha_1, \ldots, \alpha_{m-1}\in[n]^{k-1}).$$

The tensor product  possesses a very useful property: the associative law (\cite{Sh12} Theorem 1.1).

In \cite{Ch2} and \cite{Pe102}, Chang et al. and Pearson define the primitive tensors as follows.

\begin{definition} {\rm (\cite{Ch2, Pe102}) \label{defn21}} Let $\mathbb{A}$ be a nonnegative  tensor with order $m$ and dimension $n$,
$x=(x_1, x_2, \ldots, x_n)^T\in\mathbb{R}^n$ a vector and $x^{[r]}=(x_1^r, x_2^r, \ldots, x_n^r)^T$. Define the map $T_\mathbb{A}$ from $\mathbb{R}^n$ to $\mathbb{R}^n$ as: $T_\mathbb{A}(x)=(\mathbb{A}x)^{[\frac{1}{m-1}]}$. If there exists some positive integer $r$ such that $T_\mathbb{A}^r(x)>0$ for all nonnegative nonzero vectors $x\in\mathbb{R}^n$, then $\mathbb{A}$ is called primitive and the smallest such integer $r$ is called the primitive degree of $\mathbb{A}$, denoted by $\gamma(\mathbb{A})$. \end{definition}

In \cite{Sh12}, Shao shows the following results and defines the primitive degree by using the properties of tensor product and the zero patterns.

\begin{proposition}\label{pro22}{\rm(\cite{Sh12} Theorem 4.1)}
A nonnegative tensor $\mathbb{A}$ is primitive if and only if there exists some positive integer $r$ such that $\mathbb{A}^r$ is essentially positive.
Furthermore, the smallest such $r$ is the primitive degree of $\mathbb{A}$, $\gamma(\mathbb{A})$.
\end{proposition}

In \cite{YHY13}, the authors prove some necessary conditions for a nonnegative tensor to be a primitive one, and they  also prove the following result.

\begin{proposition}\label{pro23}{\rm (\cite{YHY13} Remark 2.6)}
Let $\mathbb{A}$ be a nonnegative  tensor with order $m$ and dimension $n$. Then $\mathbb{A}$ is primitive if and only if there exists some positive integer $r$ such that $M(\mathbb{A}^r)>0.$ Furthermore, the smallest such $r$ is the primitive degree of $\mathbb{A}$, $\gamma(\mathbb{A})$.
\end{proposition}

In \cite{YHY13}, the authors introduce some theoretical concepts of digraphs and matrices.

Let $D=(V,A)$ denote  a digraph on $n$ vertices. Loops are
permitted, but no multiple arcs. A $u\rightarrow v$ walk in
$D$ is a sequence of vertices $u, u_1,\ldots, u_k=v$ and a
sequence of arcs $e_1=(u,u_1),e_2=(u_1,u_2), \ldots,
e_k=(u_{k-1},v)$, where the vertices and the arcs are not
necessarily distinct. We use the notation
$u\rightarrow u_1\rightarrow u_2\rightarrow \cdots \rightarrow u_{k-1}\rightarrow v$
to refer to this  $u\rightarrow v$ walk.
A closed walk is a $u\rightarrow v$ walk where $u=v$. A path is a walk with distinct vertices. A
 cycle is a closed $u\rightarrow v$ walk with distinct
vertices except for $u=v$. The length of a walk $W$ is the
number of arcs in $W$, denoted by $l(W)$. A $k$-cycle is a cycle of length $k$, denoted by $C_k$.

\begin{definition}\label{defn24}{\rm (\cite{YHY13} Definition 2.9)}
Let $D=(V,A)$ denote  a digraph on $n$ vertices.
A digraph $D^{\prime}=(V, A^{\prime})$ is called the reversed digraph of $D$
where $(j,i)\in A^{\prime}$ if and only if $(i,j)\in A$ for any $i,j\in V$, denoted by $\overleftarrow{D}$.
\end{definition}

Let $M=(m_{ij})$ be a square nonnegative matrix of order $n$.
The associated digraph $D(M)=(V, A)$ of $M$ (possibly with loops) is defined to be the digraph with vertex set
$V=\{1,2,\ldots,n\}$ and arc set $A=\{(i,j)|m_{ij}\neq 0\}$.
The associated reversed digraph $\overleftarrow{D(M)}=(V, A^{\prime})$ of $M$ (possibly with loops) is defined to be the digraph with vertex set
$V=\{1,2,\ldots,n\}$ and arc set $A=\{(j,i)|m_{ij}\neq 0\}$.
Clearly, the associated reversed digraph of $M$ is the reversed digraph of the associated digraph of $M$.

Let $N^+_D(i)=\{j\in V(D)|(i, j)\in E(D)\}$ denote the  out-neighbors of $i$ and $d^+_i=|N^+_D(i)|$ denote the out-degree of the vertex $i$ in $D$.
The following Proposition is the graph theoretical version of Proposition 2.7 in \cite{YHY13}.

\begin{proposition}\label{pro25}{\rm (\cite{YHY13} Proposition 2.7)}
Let $\mathbb{A}$ be a nonnegative primitive tensor with order $m$ and dimension $n$, $M(\mathbb{A})$ the majorization matrix of $\mathbb{A}$.
Then in the digraph $\overleftarrow{D(M(\mathbb{A}))}=(V, A^{\prime})$, we have:

{\rm (i) } For each $j\in V$, the out-degree of the vertex $j$, $d^+_j\geq 1$ and $N^+_D(j)\ne\{j\}$.

{\rm (ii) } There exists at least a vertex $j\in V$,   the out-degree of the vertex $j$, $d^+_j\geq 2$.
\end{proposition}

The authors \cite{YHY13} also presented the definition of $j$-primitive degree for a nonnegative tensor and obtained the following result.

\begin{definition}\label{defn26}{\rm (\cite{YHY13} Definition 2.13)}
 Let $\mathbb{A}$ be a nonnegative  tensor with order $m$ and dimension $n$. For a fixed integer $j\in[n]$, if there exists a positive integer $k$ such that
$$(M(\mathbb{A}^k))_{uj}>0, \hskip.2cm {\mbox for\,\, all } \, u \in[n],$$
then $\mathbb{A}$ is called $j$-primitive and the smallest such integer $k$ is called the $j$-primitive degree of $\mathbb{A}$, denoted by $\gamma_j(\mathbb{A})$. \end{definition}

\begin{proposition}\label{pro27} {\rm (\cite{YHY13} Proposition 2.14)}
Let $\mathbb{A}$ be a nonnegative primitive  tensor with order $m$ and dimension $n$. Then
$\gamma(\mathbb{A})=\max\limits_{1\le j\le n}\{\gamma_j(\mathbb{A})\}.$\end{proposition}

\section{Necessary and sufficient conditions for  nonnegative  tensors to be primitive}

\hskip.6cm In this section, we will present necessary and sufficient conditions for a nonnegative  tensor to be a primitive one.
Let $\mathbb{A}$ be a nonnegative  tensor with order $m$ and dimension $n$. For positive integers $k$ and $j\in[n]$, we put

$$S_k(\mathbb{A}, j)=\{u\in[n]\hskip.08cm |\hskip.08cm (M(\mathbb{A}^k))_{uj}>0 \}, k=1, 2, \ldots.$$

\noindent Then $ S_1(\mathbb{A}, j)=\{u\in[n]\hskip.08cm|\hskip.08cm (M(\mathbb{A}))_{uj}>0 \}=N_{\overleftarrow{D(M(\mathbb{A}))}}^+(j)$.

 By the definitions of the  general tensor product and the majorization matrix of a tensor, we can show that
 \begin{equation}\label{eq31}(M(\mathbb{A}^{k+1}))_{uj}=\sum\limits_{j_2,\ldots, j_m=1}^{n}a_{uj_2\cdots j_m}(M(\mathbb{A}^k))_{j_2j}\cdots(M(\mathbb{A}^k))_{j_mj},\end{equation}

\noindent it follows that $u\in S_{k+1}(\mathbb{A}, j)$ if and only if there exist indices $j_2, \ldots, j_m\in S_k(\mathbb{A}, j)$  and $a_{uj_2\cdots j_m}>0$.
Thus
 \begin{equation}\label{eq32}S_{k+1}(\mathbb{A}, j)=\{u\in[n]\hskip.08cm|\hskip.08cm\mbox {there exist }  j_2, \ldots, j_m\in S_k(\mathbb{A}, j) \mbox{ and } a_{uj_2\cdots j_m}>0\}.\end{equation}

By the definition of  $S_k(\mathbb{A}, j)$, we have the following  result.
\begin{lemma}\label{lem31} Let $\mathbb{A}$ be a nonnegative  tensor with order $m$ and dimension $n$.

{\rm (i) } Let $k, l, i, j$ be positive integers such that $1\le i, j\le n$. Suppose that $S_k(\mathbb{A}, i)=S_l(\mathbb{A}, j)$, then
$S_{k+r}(\mathbb{A}, i)=S_{l+r}(\mathbb{A}, j)$ holds for every positive integer $r$.

{\rm (ii) } For any $j\in[n]$, let $k$ be the least positive integer such that $S_k(\mathbb{A}, j)=[n]$.
Then for any integer $l\geq k$,  $S_l(\mathbb{A}, j)=[n]$.
\end{lemma}

\begin{proof} (i) Since $S_k(\mathbb{A}, i)=S_l(\mathbb{A}, j)$, by (\ref{eq32}), we have

\hskip0.70cm$S_{k+1}(\mathbb{A}, i)=\{u\in[n] \hskip.08cm|\hskip.08cm \mbox {there exist }  j_2, \ldots, j_m\in S_k(\mathbb{A}, i) \mbox{ and } a_{uj_2\cdots j_m}>0\}$

\hskip2.6cm$=\{u\in[n]\hskip.08cm |\hskip.08cm \mbox {there exist }  j_2, \ldots, j_m\in S_l(\mathbb{A}, j) \mbox{ and } a_{uj_2\cdots j_m}>0\}$

\hskip2.6cm$=S_{l+1}(\mathbb{A}, j).$

\noindent Therefore (i) follows by induction on $r$.

(ii) Let $k$ be the least positive integer such that $S_k(\mathbb{A}, j)=[n]$. We complete the proof by the following two cases.

\noindent {\bf Case 1: }  $k>1$.

Then $S_{k-1}(\mathbb{A}, j)\subset[n]$, by (\ref{eq32}), we have

$[n]=S_k(\mathbb{A}, j)$

\hskip0.6cm$=\{u\in[n]\hskip.08cm | \hskip.08cm \mbox {there exist }  j_2, \ldots, j_m\in S_{k-1}(\mathbb{A}, j)\subset [n]
\mbox{ and } a_{uj_2\cdots j_m}>0\}$

\hskip0.6cm$\subseteq\{u\in[n] \hskip.08cm|\hskip.08cm \mbox {there exist }  j_2, \ldots, j_m\in S_{k}(\mathbb{A}, j)= [n]
\mbox{ and } a_{uj_2\cdots j_m}>0\}$

\hskip0.6cm $=S_{k+1}(\mathbb{A}, j)\subseteq [n]$.

\noindent Thus $S_{k+1}(\mathbb{A}, j)=[n], $  and we can show $S_{k+r}(\mathbb{A}, j)=[n]$ for any nonnegative integer $r$ by induction on $r$.




  \noindent {\bf Case 2: }  $k=1$.

 Then $S_1(\mathbb{A}, j)=[n]$, i.e., $(M(\mathbb{A}))_{uj}>0$ for all $u\in[n]$.
By taking $j_2=j_3=\cdots=j_m=j$, it is easy to show that $S_2(\mathbb{A}, j)=[n]$.
Therefore $S_{l}(\mathbb{A}, j)=[n]$ for any positive integer $l$ by induction on $l$.
\end{proof}

\begin{remark}\label{rem32}
Let $\mathbb{A}$ be a nonnegative  tensor with order $m$ and dimension $n$. Then for any $j\in[n]$, $\gamma_j(\mathbb{A})$ is the least positive integer $k$ satisfying $S_k(\mathbb{A}, j)=[n]$ by Definition \ref{defn26}.
\end{remark}

Denote $e_j$  the vector (of dimension $n$) whose $j$th component is $1$ and others are $0$ for any $j=1, \ldots, n$.
Let $x_j^{(0)}=e_j$ and for any nonnegative integer $k$, we define
 $$x^{(k+1)}_j=\mathbb{A}x^{(k)}_j.$$

 By the definition of the general tensor product by Shao, we know $x_j^{(0)}, x_j^{(1)}, \ldots $ are  vectors (of dimension $n$). Let
 $$T_{k}(\mathbb{A}, j)=\{u\in[n]\hskip.08cm |\hskip.08cm \mbox{the $u$th component of } x^{(k)}_j \mbox{ is larger than 0}\}.$$

By the definition of $x_j^{(1)}$,  for $u\in[n]$, we have
$$\left(x_j^{(1)}\right)_u=a_{uj\cdots j}=(M(\mathbb{A}))_{uj},$$
it follows that $S_1(\mathbb{A}, j)=T_1(\mathbb{A}, j)$ by the definitions of $S_1(\mathbb{A}, j)$ and $T_1(\mathbb{A}, j)$. Since
$$\left(x^{(k+1)}_j\right)_u=\sum_{i_2, \ldots, i_m=1}^na_{ui_2\cdots i_m}\left(x^{(k)}_j\right)_{i_2}\cdots\left(x^{(k)}_j\right)_{i_m}, $$
it follows that $\left(x^{(k+1)}_j\right)_u>0$ if there exist  some indices $i_2, \ldots, i_m\in T_k(\mathbb{A}, j)$ and  $a_{ui_2\cdots i_m}>0$, thus $S_{k+1}(\mathbb{A}, j)=T_{k+1}(\mathbb{A}, j)$. Now by induction on $k$ and  the definitions of $S_k(\mathbb{A}, j)$ and $T_k(\mathbb{A}, j)$, we see that $S_k(\mathbb{A}, j)=T_k(\mathbb{A}, j)$ for all $k$ and $j\in[n]$. Therefore, we have

 \begin{theorem}\label{thm33}
 Let $\mathbb{A}$ be a nonnegative tensor with order $m$ and dimension $n$. For a fixed integer $j\in[n]$,  $\mathbb{A}$ is  $j$-primitive if and only if
 there exists some positive  integer $k$ such that $x^{(k)}_j>0$  and the smallest such integer $k$ is  the $j$-primitive degree $\gamma_j(\mathbb{A})$.\end{theorem}

\begin{theorem}\label{thm34}
Let $\mathbb{A}$ be a nonnegative  tensor with order $m$ and dimension $n$. Then $\mathbb{A}$ is primitive if and only if
 there exists some positive  integer $k$ such that $x^{(k)}_j>0$ for all $j\in[n]$.
 Furthermore,  the smallest such integer $k$ is  the primitive degree $\gamma(\mathbb{A})$ and $\gamma(\mathbb{A})\le (n-1)^2+1$.\end{theorem}
\begin{proof} By Proposition \ref{pro23}, Definition \ref{defn26} and Theorem \ref{thm33}, we know that

\hskip.3cm $\mathbb{A}$ is primitive

\noindent $\Longleftrightarrow$ there exists some positive $k$ such that $M(\mathbb{A}^k)>0$

\noindent $\Longleftrightarrow$ for all $j\in [n]$, there exists some positive $k$ such that $(M(\mathbb{A}^k))_{uj}>0$ for all $u\in [n]$

\noindent $\Longleftrightarrow$ for all $j\in [n]$, $\mathbb{A}$ is $j$-primitive

\noindent $\Longleftrightarrow$ for all $j\in [n]$, there exists some positive  integer $k$ such that $x^{(k)}_j>0$.

Thus by Theorem \ref{thm11}, Proposition \ref{pro27}, Theorem \ref{thm33} and the above arguments,
we  know that the smallest such integer $k$ is  the primitive degree $\gamma(\mathbb{A})$ and $\gamma(\mathbb{A})\le (n-1)^2+1$.
\end{proof}

\section{Proof of the main result}
\hskip.6cm By the relation between matrices and digraphs, we know that

$(A^k)_{uj}>0 \Longleftrightarrow \mbox{ there exists a walk of length $k$ from $u$ to $j$ in the digraph } D(A)$

\hskip2cm $\Longleftrightarrow \mbox{ there exists a walk of length $k$ from $j$ to $u$ in the digraph } \overleftarrow{D(A)}.$

\noindent Then we obtain the following proposition.
\begin{proposition}\label{pro41}
Let $A$ be a nonnegative matrix of order $n$,  $j\in [n], k$ be positive integers,
$S_k(A,j)=\{u\in [n]  \hskip0.08cm|\hskip0.08cm (A^k)_{uj}>0\}$.
Then $$S_k(A,j)=\{u\in [n] \hskip0.08cm|\hskip0.08cm \mbox{there exists a walk of length $k$ from $j$ to $u$ in the digraph } \overleftarrow{D(A)}\}.$$
\end{proposition}

\begin{lemma}\label{lem42}
Let $\mathbb{A}$ be a nonnegative   tensor with order $m$ and dimension $n$ such that
  $a_{ii_2\cdots i_m}=0$ if $i_2\cdots i_m\not=i_2\cdots i_2$  for any $i\in [n]$.
  Then for any positive integers $j\in [n]$ and $k$,
  \begin{equation}\label{eq41}
  S_k(\mathbb{A}, j)=S_k(M(\mathbb{A}),j).
  \end{equation}
\end{lemma}
\begin{proof}
We prove 
$S_k(\mathbb{A}, j)=S_k(M(\mathbb{A}),j)$
by induction on $k$.  Clearly, $k=1$ is obvious by the fact
$ S_1(\mathbb{A}, j)=\{u\in[n]\hskip0.08cm|\hskip0.08cm (M(\mathbb{A}))_{uj}>0 \}=N_{\overleftarrow{D(M(\mathbb{A}))}}^+(j)=S_1(M(\mathbb{A}),j).$

Assume that (\ref{eq41}) holds for $k=l\geq 1$. Then by (\ref{eq31}), (\ref{eq32}) and Proposition \ref{pro41}, we have

 $S_{l+1}(\mathbb{A}, j)$$=\{u\in[n]\hskip0.08cm|\hskip0.08cm (M(\mathbb{A}^{l+1}))_{uj}>0 \}$

\noindent $=\{u\in[n]\hskip0.08cm|\hskip0.08cm\mbox {there exist }  j_2, \ldots, j_m\in S_l(\mathbb{A}, j) \mbox{ and } a_{uj_2\cdots j_m}>0\}$

\noindent $=\{u\in[n]\hskip0.08cm|\hskip0.08cm\mbox {there exist }  v\in S_l(\mathbb{A}, j) \mbox{ and } a_{uv\cdots v}>0\}$

\noindent $=\{u\in[n]\hskip0.08cm|\hskip0.08cm\mbox {there exist }  v\in S_l(M(\mathbb{A}), j) \mbox{ and } (M(\mathbb{A}))_{uv}>0\}$

\noindent $=\{u\in [n]\hskip0.08cm |\hskip0.08cm \mbox{there exists a walk of length $l$ from $j$ to $v$  and arc $(v,u)$ in  } \overleftarrow{D(M(\mathbb{A}))}\}.$

\noindent $=\{u\in [n] \hskip0.08cm|\hskip0.08cm \mbox{there exists a walk of length $l+1$ from $j$ to $u$  in  } \overleftarrow{D(M(\mathbb{A}))}\}.$

\noindent $=S_{l+1}(M(\mathbb{A}),j)$.

Thus, for any positive integers $j\in [n]$ and $k$, $S_k(\mathbb{A}, j)=S_k(M(\mathbb{A}),j)$ holds.
\end{proof}

\begin{lemma}\label{lem43}{\rm (\cite{YHY13} Corollary 3.4)}
Let $\mathbb{A}$ be a nonnegative primitive  tensor with order $m$ and dimension $n$ such that
  $a_{ii_2\cdots i_m}=0$ if $i_2\cdots i_m\not=i_2\cdots i_2$  for any $i\in [n]$.
  If $M(\mathbb{A})$ is primitive, then
   $\gamma(\mathbb{A})=\gamma(M(\mathbb{A})).$
\end{lemma}
Let $M_1=\left(
                 \begin{array}{ccccc}
                   0 & 0  & \cdots & 1 & 1 \\
                   1 & 0   & \cdots &  0&  0\\
                    0 & 1  & \cdots &  0& 0  \\
                    0 & 0 & \ddots &  0&0  \\
                   0 &0 & \cdots & 1 & 0 \\
                 \end{array}
               \right)
$. It is well known that $M_1$ is primitive, and $\gamma(M_1)=(n-1)^2+1$.

\vskip-.7cm
$$
        \hskip1cm
 \xy 0;/r3pc/: \POS (1,1) *\xycircle<3pc,3pc>{};
        \POS(1,2) *@{*}*+!D{n}="n";
        \POS(1.5,1.86) \ar@{->}(1.5,1.86);(1.6,1.8)="a";
        \POS(.5,1.86)   \ar@{->}(.5,1.86);(.6,1.91)="b";
        \POS(1.8,1.6)  *@{*}*+!L{\hspace*{3pt}{n\hspace*{-3pt}-\hspace*{-3pt}1}}="c";
        \POS(.2,1.6)   *@{*}*+!R{\mathrm{1}}="d";
        \POS(.2,.4)    \ar@{->}(.2,.4) ;(.19,.415) ="e";
        \POS(1.8,.4)   \ar@{->}(1.8,.4);(1.79,.385)="f";
        \POS "c" \ar @{->} (.7,1.6) \ar @{-} "d";
         \POS (0.015, .9) *@{*}*+!R{2}="g";
         \POS(.4,0.2) *@{*}*+!R{\mathrm{}}="k";
      \POS(.85,0.006) \ar@{->}(1,0.00);(0.85,.006)="r";
         \POS(1.6,0.2) *@{*}*+!R{\mathrm{}}="l";
           \POS(0.8,0.2) *@{*}*+!R{\mathrm{}}="m";
             \POS(1.0,0.2) *@{*}*+!R{\mathrm{}}="p";
               \POS(1.2,0.2) *@{*}*+!R{\mathrm{}}="q";

         \POS(2.0,.9)  *@{*}*+!L{\hspace*{3pt}{n\hspace*{-3pt}-\hspace*{-3pt}2}}="h";
        \POS(1.99,1.1) \ar@{->}(1.98,1.25);(1.99,1.1)="i";
        \POS(.01,1.1)   \ar@{->}(.01,1.1);(.023,1.25)="j";

 \endxy
 $$
 $$\mbox{ \qquad Figure 1.  digraph } \overleftarrow{D(M_1)}$$

Let  $\mathbb{A}_0$ be the nonnegative primitive  tensor with order $m$ and dimension $n$ such that
  $a_{ii_2\cdots i_m}=0$ if $i_2\cdots i_m\not=i_2\cdots i_2$  for any $i\in [n]$, and $M(\mathbb{A}_0)=M_1$.
   Then by Lemma \ref{lem42}, we have $S_k(\mathbb{A}_0, j)=S_k(M_1,j)$ for any $j\in [n]$ and any positive integer $k$,
   and  by Lemma \ref{lem43}, we have $\gamma(\mathbb{A}_0)=(n-1)^2+1$.

The following result is well known.

\begin{proposition}\label{pro44}
Let $a,b$ be positive integers,  if $a, b$ are coprime $(g.c.d.(a,b)=1)$,
then equation $ax+by=ab-a-b$ has no nonnegative integral solutions $(x, y)$.
\end{proposition}

  \begin{proposition}\label{pro45}
  Let  $\mathbb{A}_0$ be the nonnegative primitive  tensor with order $m$ and dimension $n$ defined as above. Then
   $\gamma_{n-1}(\mathbb{A}_0)=n^2-3n+3$.
   \end{proposition}
   \begin{proof}
  By Remark \ref{rem32}, Proposition \ref{pro41} and Lemma \ref{lem42},  we know

  \noindent\hskip.4cm $\gamma_{n-1}(\mathbb{A}_0)=\min\{k\hskip.08cm |\hskip.08cm S_k(\mathbb{A}_0,n-1)=[n]\}$

 \noindent $=\min\{k\hskip.08cm |\hskip.08cm S_k(M_1,n-1)=[n]\}$

\noindent $=\min\{k\hskip.08cm |\hskip0.08cm \mbox{there exists a walk of length $k$ from $n-1$ to $u$ in  } \overleftarrow{D(M_1)} \mbox { for all } u\in [n]\}$.

Let $W$ be any  walk of length $n^2-3n+2$ from vertex $n-1$ to vertex $n$  in $\overleftarrow{D(M_1)}$.
Then $W$ is a ¡°union¡± of the unique path $P$ from $n-1$ to $n$ (of length 1) and several cycles of length $n-1$ and several cycles of length $n$.
Let $l(W)$ be the length of $W$. Then there exist two nonnegative integer $a,b$ such that
 $$n^2-3n+2=l(W)=1+na+(n-1)b, \hskip0.1cm a\geq0,\hskip0.1cm b\geq 0.$$

Note that $n, n-1$ are coprime, by Proposition \ref{pro44}, we know that equation $n^2-3n+1=nx+(n-1)y$ has no nonnegative integral solutions $(x, y)$. It is a contradiction.
Therefore there does not exist a walk of length $n^2-3n+2$ from $n-1$ to $n$ in $\overleftarrow{D(M_1)}$,
it implies  $S_{n^2-3n+2}(M_1,n-1)\not=[n]$ and thus $\gamma_{n-1}(\mathbb{A}_0)>n^2-3n+2$.

Let $u\in [n]$ be any vertex in $\overleftarrow{D(M_1)}$, $P$ be the path of length $l = l(P)$ from vertex $n-1$ to vertex $u$, then
$$l=\left\{\begin{array}{ll}
     0, & \mbox { if } u=n-1; \\
     1, & \mbox { if } u=n; \\
     u \mbox { or } u+1,  & \mbox { if } u=1, 2,  \ldots, n-2.
   \end{array}\right.
$$

Let $C_{n-1}$ and $C_n$ be the cycles of length $n-1$ and $n$ in $\overleftarrow{D(M_1)}$, respectively. Take
$$W=\left\{\begin{array}{ll}
     (n-3)C_{n-1}+C_n, & \mbox { if } u=n-1; \\
     1+(n-2)C_{n-1}, & \mbox { if } u=1 \mbox { or } u=n; \\
     l+(l-3)C_{n-1}+(n-l)C_n,  & \mbox { if } u=2, 3, \ldots, n-2, \mbox { where } 3\leq l\leq n-1.
   \end{array}\right.
$$

Clearly, the length of $W$, $l(W)=n^2-3n+3$.
Therefore there exists a walk of length $n^2-3n+3$ from $n-1$ to any $u\in [n]$ in $\overleftarrow{D(M_1)}$,
it implies  $S_{n^2-3n+3}(M_1,n-1)=[n]$ and thus $\gamma_{n-1}(\mathbb{A}_0)\leq n^2-3n+3$.

Combining the above arguments, we have $\gamma_{n-1}(\mathbb{A}_0)= n^2-3n+3$.
   \end{proof}

\begin{remark}\label{rem46}
Similar to the proof of Proposition \ref{pro45}, take
$$W=\left\{\begin{array}{ll}
     (n-2)C_{n-1}, & \mbox { if } u=n-1; \\
     l+(l-2)C_{n-1}+(n-l-1)C_n,  & \mbox { if } u=1, 2, \ldots, n-2, \mbox { where } 2\leq l\leq n-1.
   \end{array}\right.
$$
We can show that there  exists a walk of length $n^2-3n+2$ from $n-1$ to any $u\in [n-1]$ in $\overleftarrow{D(M_1)}$,
it implies  that $S_{n^2-3n+2}(M_1,n-1)=[n-1]$  and  there does not exist a walk of length $n^2-3n+2$ from $n-1$ to $n$ in $\overleftarrow{D(M_1)}$.
\end{remark}

Note that $S_1( \mathbb{A}_0, n-1)=\{1,n\}$, then for any  positive integer $k$ with $1 \le k \le n^2-3n+2$,
we can show that $2\le |S_k(\mathbb{A}_0, n-1 )|=|S_k(M_1, n-1 )|\leq n-1$ by Proposition \ref{pro41}, Proposition \ref{pro45} and Remark \ref{rem46}.

  \begin{proposition}\label{pro47}{\rm (\cite{Sh12} Corollary 4.1 )}
   Let $\mathbb{A}$  be a nonnegative tensor with order $m$ and dimension $n$. If $M(\mathbb{A})$ is primitive, then $\mathbb{A}$ is also primitive.
   \end{proposition}

 \begin{proposition}\label{pro48}
 Let  $m, n, k$ be positive integers with $m\ge n\ge 3$ and $1 \le k \le n^2-3n+2$,
$\mathbb{A}_k$  be  the nonnegative  tensor with order $m$ and dimension $n$ such that
 for any $i\in[n]$,
  $$(\mathbb{A}_k)_{ii_2\cdots i_m}=\left\{\begin{array}{ll}
                                            1, & \mbox { if } \{i_2, \ldots, i_m\}=S_k(\mathbb{A}, n-1);  \\
                                            (\mathbb{A}_0)_{ii_2\cdots i_m}, & otherwise.
                                          \end{array}\right.$$
Then  $M(\mathbb{A}_k)=M(\mathbb{A}_0)=M_1$ and $\mathbb{A}_k$ is  primitive.
    \end{proposition}
    \begin{proof}
  Clearly, $\mathbb{A}_k$ is well-defined since $m\ge n$ and $|S_k(\mathbb{A}, n-1)|\le n-1$.

Since $|S_k(\mathbb{A}, n-1)|\ge 2$ for any $k$,
then $i_2\ldots i_m\not=i_2\ldots i_2$ when $\{i_2, \ldots, i_m\}=S_k(\mathbb{A}, n-1)$,
and thus $M(\mathbb{A}_k)=M(\mathbb{A}_0)=M_1$ by the definition of the majorization matrix of a tensor.

Note that $M_1$ is primitive, then  $\mathbb{A}_k$ is  primitive by Proposition \ref{pro47}.
\end{proof}

\begin{proposition}\label{pro49}
Let $k$ be positive integer with $1\leq k\le n^2-3n+2$. Then
 $S_1(\mathbb{A}_0, n-1), S_2(\mathbb{A}_0, n-1), \ldots, S_k(\mathbb{A}_0, n-1)$
are pairwise distinct proper subsets of $[n]$.
\end{proposition}
\begin{proof}
Clearly, $S_1(\mathbb{A}_0, n-1), S_2(\mathbb{A}_0, n-1), \ldots, S_k(\mathbb{A}_0, n-1)$
are proper subsets of $[n]$ by $2\le |S_p(\mathbb{A}_0, n-1 )|\leq n-1$ for all $p\in [k]$.

If $S_p(\mathbb{A}_0, n-1)= S_q(\mathbb{A}_0, n-1)$ for $1\leq p<q\le k$,
then by (i), (ii) of Lemma \ref{lem31} and Proposition \ref{pro45},
we have $$S_{n^2-3n+2}(\mathbb{A}_0, n-1)= S_{q+(n^2-3n+2-p)}(\mathbb{A}_0, n-1)=S_{n^2-3n+3}(\mathbb{A}_0, n-1)=[n],$$
but $S_{n^2-3n+2}(\mathbb{A}_0, n-1)=[n-1]\not=[n]$ by Remark \ref{rem46}, it is a contradiction.
Thus $S_1(\mathbb{A}_0, n-1), S_2(\mathbb{A}_0, n-1), \ldots, S_k(\mathbb{A}_0, n-1)$
are pairwise distinct.
\end{proof}

\begin{theorem}\label{thm410}
Let  $m, n, k$ be positive integers with $m\ge n\ge 3$ and $1 \le k \le n^2-3n+2$,
 $\mathbb{A}_k$ be the nonnegative primitive  tensor defined as above.
 Then $\gamma(\mathbb{A}_k)=\gamma_{n}(\mathbb{A}_k)=n+k$.\end{theorem}

\begin{proof}
Firstly,  we show that  $S_t(\mathbb{A}_k, n-1)=S_t(\mathbb{A}_0, n-1)$ for any positive integer $t\in [k]$ by induction on $t$.

By Proposition \ref{pro48}, we have  $S_1(\mathbb{A}_k, n-1)=\{u\in[n] \hskip.08cm | \hskip.08cm (M(\mathbb{A}_k))_{u, n-1}>0\}
                                  =\{u\in[n] \hskip.08cm | \hskip.08cm (M(\mathbb{A}_0))_{u, n-1}>0\}
                                  =S_1(\mathbb{A}_0, n-1)=\{1,n\}$. Then $t=1$ holds.

Assume that  $S_t(\mathbb{A}_k, n-1)=S_t(\mathbb{A}_0, n-1)$ for $t=l\geq 1$ holds. Then for $t=l+1\leq k $,
$S_l(\mathbb{A}_0, n-1)\not=S_k(\mathbb{A}_0, n-1)$ by Proposition \ref{pro49}.
If $i_2, \ldots, i_m\in S_l(\mathbb{A}_0, n-1)$,
we have $(\mathbb{A}_k)_{ui_2\ldots i_m}>0\Longleftrightarrow
(\mathbb{A}_k)_{ui_2\ldots i_m}=(\mathbb{A}_0)_{ui_2\ldots i_m}>0$ by the definition of $\mathbb{A}_k$. Thus by (\ref{eq32}),

$S_{l+1}(\mathbb{A}_k, n-1)
=\{u\in[n] \hskip.08cm | \hskip.08cm \mbox{there exist } i_2, \ldots, i_m\in S_l(\mathbb{A}_k, n-1) \mbox{ and } (\mathbb{A}_k)_{ui_2\ldots i_m}>0\}$

 \hskip2.85cm $ =\{u\in[n] \hskip.08cm | \hskip.08cm \mbox{there exist } i_2, \ldots, i_m\in S_l(\mathbb{A}_0, n-1) \mbox{ and } (\mathbb{A}_k)_{ui_2\ldots i_m}>0\}$

 \hskip2.85cm $ =\{u\in[n] \hskip.08cm | \hskip.08cm \mbox{there exist } i_2, \ldots, i_m\in S_l(\mathbb{A}_0, n-1) \mbox{ and } (\mathbb{A}_0)_{ui_2\ldots i_m}>0\}$

  \hskip2.85cm  $=S_{l+1}(\mathbb{A}_0, n-1)$.

\noindent It follows that $S_t(\mathbb{A}_k, n-1)=S_t(\mathbb{A}_0, n-1)$ for any positive integer $t$ with  $1\le t\le k$,
and thus $S_1(\mathbb{A}_k, n-1), S_2(\mathbb{A}_k, n-1), \ldots, S_k(\mathbb{A}_k, n-1)$
are pairwise distinct proper subsets of $[n]$ for $1\le k\le n^2-3n+2$ by Proposition \ref{pro49}.

Since $(\mathbb{A}_k)_{ii_2\cdots i_m}=1$   for any $i\in [n]$ when the set $\{i_2, \ldots, i_m\}=S_k(\mathbb{A}_k, n-1)$, we have
$$S_{k+1}(\mathbb{A}_k, n-1)=\{u\in[n]\hskip.08cm |\hskip.08cm \mbox{there exist } j_2, \ldots, j_m\in S_k(\mathbb{A}, n-1)
 \mbox{ and } (\mathbb{A}_k)_{uj_2\cdots j_m}>0\}=[n].$$
and thus $\gamma_{n-1}(\mathbb{A}_k)=k+1$ by Remark \ref{rem32}.

By the definition of $\mathbb{A}_k$, an easy computation shows that
$$S_1(\mathbb{A}_k, n-1)=S_2(\mathbb{A}_k, n-2)=S_3(\mathbb{A}_k, n-3)=\cdots =S_{n-1}(\mathbb{A}_k, 1)=S_n(\mathbb{A}_k, n).$$
By (i) of Lemma \ref{lem31} and the definition of $j$-primitive degree, we have

$S_1(\mathbb{A}_k, n-1)=S_2(\mathbb{A}_k, n-2)$

\noindent $\Longrightarrow
\left\{\begin{array}{ll}
   ([n]\not=)S_{1+r}(\mathbb{A}_k, n-1)=S_{2+r}(\mathbb{A}_k, n-2), & \mbox{ if } 1\le r\le k-1; \\
    ([n]=)S_{1+k}(\mathbb{A}_k, n-1)=S_{2+k}(\mathbb{A}_k, n-2), &     \mbox{ if } r=k.
      \end{array}\right.$

\noindent $\Longrightarrow \gamma_{n-2}(\mathbb{A}_k)=k+2$.

Similarly, we can prove $\gamma_{n-3}(\mathbb{A}_k)=k+3,  \ldots,  \gamma_{1}(\mathbb{A}_k)=n+k-1, \gamma_{n}(\mathbb{A}_k)=n+k.$
Then  following form Proposition \ref{pro27}, we have $\gamma(\mathbb{A}_k)=\max\limits_{1\le j\leq n}\gamma_{j}(\mathbb{A}_k)=\gamma_{n}(\mathbb{A}_k)=n+k$.\end{proof}

\noindent {\bf Proof of Theorem \ref{thm12}:}
Let $t$ be any positive integer with $1\le t\le (n-1)^2+1$,
we will show that there exists a nonnegative primitive tensor  $\mathbb{B}_t$  with order $m$ and dimension $n$ such that $\gamma(\mathbb{B}_t)=t$.
We complete the proof by the following two cases.

{\bf Case 1: }  $1\le t\le n$.

It is well known that there exists a primitive matrix $A_t$ of order $n$ such that $\gamma(A_t)=t$.
We define the tensor $\mathbb{B}_t$
to be  the nonnegative primitive  tensor with order $m$ and dimension $n$ such that
  $(\mathbb{B}_t)_{ii_2\cdots i_m}=0$    if $i_2\cdots i_m\not=i_2\cdots i_2$  for any $i\in [n]$, and $M(\mathbb{B}_t)=A_t$.
  Then $\gamma(\mathbb{B}_t)=\gamma(A_t)=t$ by Lemma \ref{lem43}.

 {\bf Case 2: } $n+1\le t\le (n-1)^2+1$.

 We choose $\mathbb{B}_t=\mathbb{A}_{t-n}$. Then $\gamma(\mathbb{B}_t)=n+(t-n)=t$ by Theorem \ref{thm410}. \qed



\end{document}